\documentclass[reqno]{amsart}

    \usepackage{float}
    \usepackage{amsmath}
    \usepackage{graphicx}
    \usepackage{latexsym}
    \usepackage{amsfonts}
    \usepackage{amssymb}
    \usepackage{verbatim}
 
    \usepackage[hyperindex=true]{hyperref}

    \theoremstyle{plain}
    \newtheorem{theorem}{Theorem}
    \newtheorem{corollary}[theorem]{Corollary}

    \theoremstyle{definition}
    \newtheorem{assumption}{Assumption}

    \newtheorem{remark}[theorem]{Remark}
    \newtheorem*{remark*}{Remark}

    \newcommand{\pr}{\mathbf P}

    \newcommand{\abs}[1]{\left\lvert#1\right\rvert}

    \makeatletter
    \@namedef{subjclassname@2020}{\textup{2020} Mathematics Subject Classification}
    \makeatother

\begin{document}
    \title[A weak invariance principle for triangular arrays ]
    {A weak invariance principle for triangular arrays of independent random 
variables in some Besov spaces}

    \author[Giraudo]{Davide Giraudo}
    \address{Institut de Recherche Math\'ematique Avanc\'ee (IRMA), UMR~7501,
Universit\'e de Strasbourg and CNRS,
7~rue Ren\'e Descartes, 67000~Strasbourg, France}
    \email{dgiraudo@unistra.fr.}

    \author[Sharipov]{Sadillo Sharipov}
    \address{V.I.Romanovskiy Institute of Mathematics, Uzbekistan Academy of Sciences, Tashkent, Uzbekistan}
    \email{sadi.sharipov@yahoo.com}

\begin{abstract}
We establish a Donsker--Prokhorov invariance principle in some Besov spaces. Specifically, we show that polygonal line processes associated with partial sum processes of a triangular array of row-wise independent random variables converge in distribution to Brownian motion, extending earlier results in the literature.
\end{abstract}


    \maketitle

\section{Introduction}\label{sec1}

Suppose that for each $n\geq 1$, there is given a sequence of independent random variables
\begin{equation*} 
X_{n,1}, X_{n,2}, \dots, X_{n,m_{n}}
\end{equation*}
with $\mathbf{E}X_{n,j}=0$, $\sum_{i=1}^{m_{n}}\sigma_{n,i}^{2}=1$ where $\sigma_{n,j}^{2}= \operatorname{Var}(X_{n,j})\in (0,\infty)$ for $j=1,\dots, m_{n}$, $n\geq 1$.

Denote $a_{n,0}:=0$, $a_{n,j}:=\sum_{i=1}^{j}\sigma_{n,i}^{2}$ (then $a_{n,m_{n}}=1$),
$$ S_{n,0}:=0, \ \ S_{n,j}:=\sum_{i=1}^jX_{n,i}.$$

Consider the process $\mathbb{S}_{n}=\left(S_{n}(t),\, t\in [0,1]\right)$ defined by affine interpolation
between the points $\left(a_{n,k}, S_{n,k}\right)$, $k=0,\dots,m_n$, namely,
\begin{equation}\label{eq2}
S_{n}(t)
= S_{n,k-1} + \frac{t-a_{n,k-1}}{\sigma_{n,k}^{2}}\,X_{n,k},
\qquad
a_{n,k-1}\leq t \leq a_{n,k},
\end{equation}
for $k=1,\dots,m_n$.

The statement on weak convergence of the distributions of random polygonal
lines \eqref{eq2} to the Wiener measure is often referred to as the
\emph{Donsker--Prokhorov invariance principle}.
M.~Donsker \cite{Donsker} studied the case of independent and identically distributed (i.i.d.) random
variables $X_{1}, X_{2}, \ldots$
with zero mean and unit variance.
Setting $S_{n}=X_{1}+\cdots+X_{n}$, $n \ge 1$, $S_{0}:=0$,
he considered random polygonal lines with vertices
$\left(i/n, S_{i}/n \right)$,
$i =0,\ldots, n$.
The general version of this result, formulated under the Lindeberg condition
and including Donsker's case as a special one, was proved by Yu.
Prokhorov \cite{Prokhorov56}.

When these processes are obtained by
polygonal interpolation, weak convergence is usually considered in the space of continuous functions $C[0,1]$ equipped with the supremum norm $\left\| \cdot \right\|_{\infty}$.
Since both partial sum processes and Brownian motion exhibit finer path 
regularity, it is natural to study invariance principles in stronger function 
spaces. This observation motivates us to study an invariance principle in 
certain Besov spaces.

The Besov spaces were first introduced by O. Besov in \cite{Besov59}, where he defined a family of function spaces using moduli of smoothness and approximation properties, in connection with embedding and extension theorems. In this paper, we focus on a subclass of Besov spaces characterized by two parameters, which we now define.
For $p\geq 1$, we denote by $\mathbb{L}_{p}\left([0,1]\right)$ the space of (equivalence classes for the almost everywhere equality) of functions $x\colon [0,1]\to \mathbb{R}$ such that
$ \int_{0}^{1}\left|x(t)\right|^{p}\mathrm{d}t<\infty$, equipped with the norm 
$\left\|\cdot \right\|_{p}$.

For $x\in \mathbb{L}_{p}([0,1])$, the modulus of regularity is defined by
\begin{equation*}
\omega_p(x,\delta):=\sup_{\left|h\right|\leq \delta}\left(\int_{I_h}\left|x\left(u+h \right)-x(u)\right|^{p} \mathrm{d}u\right)^{1/p}, \ \ 0< \delta < 1,
\end{equation*}
where $I_{h}=[0,1]\cap [-h,1-h]$.

For $\alpha\in ]0,1]$ and $p\geq 1$, the Besov space
$B_{p,\alpha}^{o}([0,1])$ is defined by
\begin{equation*}
B_{p,\alpha}^{o}([0,1]):=\left\{x\in \mathbb{L}_{p}([0,1]): \lim_{\delta\to 0}\delta^{-\alpha}\omega_{p}\left(x,\delta\right)=0  \right\}.
\end{equation*}
This space is equipped with the norm
\begin{equation*}
\|{x}\|_{B_{p,\alpha}^{o}([0,1])}=\|{x}\|_{p}+\sup_{0<\delta<1}\delta^{-\alpha}\omega_{p} \left(x,\delta\right), \ \ x\in B_{p,\alpha}^{o}([0,1]).
\end{equation*}
which turns $B_{p,\alpha}^{o}([0,1])$ into a separable Banach space. Moreover, for
$0<\beta<\alpha$, the embedding $B_{p,\alpha}^{o}([0,1])\hookrightarrow 
B_{p,\beta}^{o}([0,1])$ is continuous. For $1\leq q<p$, the embedding  
$B_{p,\alpha}^{o}([0,1])\hookrightarrow
B_{q,\alpha}^{o}([0,1])$ is also continuous. It is worth noting that
H\"older spaces arise as the limiting case $p=\infty$.

The main purpose of the present paper is to establish an invariance principle 
for partial sum-process formed from triangular array of row-wise independent 
random variables. Besov spaces $B^{o}_{p,\alpha}([0,1])$ provide a finer scale 
of regularity, combining local smoothness of order
$\alpha$ with $\mathbb{L}_{p}$-integrability. For $p\geq 1$,
$ 0< \alpha < 1/2$,
we establish convergence in distribution of $\mathbb{S}_{n}$ to Wiener process 
in $B^{o}_{p,\alpha}([0,1])$.
The proof relies on a Besov-specific tightness criterion based on dyadic 
decompositions and inspired by the H\"olderian approach of Ra\c{c}kauskas and 
Suquet~\cite{RS03},
the extension of which to Besov spaces is not straightforward. Indeed, the 
$\ell^p$ norm of sequence of increments are involved, which cannot be handled 
directly by union bounds and a truncation argument is needed. Moreover, the 
tightness criterion given in Theorem~10 in \cite{GR17} applies only to 
stationary sequences with a partial sum process built on the intervals 
$((k-1)/n,k/n)$. The case $p\alpha\leq 1$ also requires different arguments as 
stationarity was used in \cite{GR17} as well as the fact that the intervals 
where the partial sum process is affine have constant length.

Our result extends corresponding result by Giraudo and 
Ra\v{c}kauskas~\cite{GR17} obtained for i.i.d.\ sequences.

The paper is organized as follows. In Section 2, we provide our results, and 
Section 3 is devoted to the Besov spaces properties and a tightness criterion. 
Section 4 contains proofs of 
our results.  

\section{Main results}
This section is devoted to the statements of the main results. We will consider 
separately the case where $p\alpha >1$  and $p\alpha\leq 1$.

To state our main result in the case $p\alpha >1$, we introduce the truncated 
random variables. For each $\tau>0$,
we define
$$ X_{n,k,\tau}:=X_{n,k}\mathbf{1}\{\left|X_{n,k
} 
\right|\leq \tau \sigma_{n,k}^{2(\alpha-1/p)} \}, \ \ k=1,\dots, m_{n}, \ \ 
n\geq 1.$$
\begin{theorem}\label{thm:BesovThm1}
Let $p>2$ and $1/p<\alpha<1/2$, and define
\begin{equation}\label{eq3}
q(p,\alpha):=\frac{1}{1/2-\alpha+1/p}.
\end{equation}
Assume that the Lindeberg condition holds, that is, 
\begin{equation}\label{eq4}
\text{for each }\varepsilon>0,\qquad\lim_{n\to\infty}\sum_{k=1}^{m_n}
\mathbf{E}\left[X_{n,k}^2\,\mathbf{1}\{|X_{n,k}|\ge\varepsilon\}\right]=0.
\end{equation}
Assume moreover that the following two conditions are satisfied:
\begin{equation}\label{eq5}
\text{for each }\varepsilon>0,\qquad
\lim_{n\to\infty}\sum_{k=1}^{m_n}
\mathbf{P}\left(|X_{n,k}|\ge
\varepsilon\,\sigma_{n,k}^{2(\alpha-1/p)}\right)=0,
\end{equation}
and, for some $q>q(p,\alpha)$,
\begin{equation}\label{eq6}
\lim_{\tau\to0}\limsup_{n\to\infty}
\sum_{k=1}^{m_n}
\sigma_{n,k}^{-2(\alpha-1/p)q}
\mathbf{E}\bigl|X_{n,k,\tau}\bigr|^{q}=0.
\end{equation}
Then, as $n\to\infty$,
\begin{equation}\label{eq7}
\mathbb{S}_{n} \xrightarrow{\mathcal{D}} \mathbb{W} \qquad\text{in the space}\ \ B_{p,\alpha}^{o}([0,1]),
\end{equation}
where $\mathbb{W}=\bigl(W(t),\,t\in[0,1]\bigr)$ is a standard Wiener process 
and $\xrightarrow{\mathcal{D}}$ denotes the convergence in distribution in the 
mentioned space.
\end{theorem}

\begin{remark}
Since $B_{\infty,\alpha}^{o}([0,1])$ coincides with H\"older spaces 
$\mathcal{H}_{\alpha}^{o}([0,1])$, Theorem~\ref{thm:BesovThm1} extends and 
complements the weak invariance principle established in a H\"olderian framework 
in~\cite{RS03}.
\end{remark}
\begin{remark}
It is worth noting that conditions \eqref{eq5} and \eqref{eq6} of Theorem 1 are satisfied under more general condition.
Indeed, one may replace these assumptions by the stronger non-truncated moment condition: there exists $q>q(p,\alpha)$ (where $q(p,\alpha)$ is defined by \eqref{eq3}) such that
\begin{equation*}
\mathbf{E}|X_{n,k}|^{q} < \infty
\quad\text{and}\quad
\sum_{k=1}^{m_n}\sigma_{n,k}^{-2\beta q}\,\mathbf{E}|X_{n,k}|^{q} \to 0, \ \ n\to \infty.
\end{equation*}
\end{remark}
The case $p\alpha\leq 1$ is, as in the case of stationary martingale difference 
sequences, special since only a finite second moment and the Lindeberg 
condition are required.
\begin{theorem}\label{thm:WIP_p_alpha_leq_1}
Let $p\geq 1$ and let $\alpha\in [0,1/2)\cap [0,1/p]$. Suppose that 
the Lindeberg condition \eqref{eq4} is satisfied.
Then, as $n\to\infty$, the convergence \eqref{eq7} holds.
\end{theorem}
The following corollary is motivated by the asymptotic behavior of triangular 
arrays whose rows are given by linear combinations of a common i.i.d.\ sequence 
with deterministic coefficients.

\begin{corollary}\label{cor:Besov-Theorem4-Lindeberg}
Let $\left(X_{k}\right)_{k \ge 1}$ be a sequence of i.i.d.\ random variables with zero mean and $\operatorname{Var}(X_{1})=1$.
For each $n \geq 1$, let us given the triangular array of random variables 
$\left(c_{n,k}X_{k}\right)_{1 \le k \le m_{n}}$, where $(c_{n,k})_{1 \le k\le 
m_{n}}$ is a triangular array of non-zero real numbers such that
$ \sum_{k=1}^{m_n} c_{n,k}^{2}=1$.
Suppose that $p>2$, $1/p <\alpha < 1/2$, and $(m_n)$ is non-decreasing with
\begin{equation}\label{eq:condition_on_mn}
 \sup_{n\ge1}\frac{m_{n+1}}{m_n}<\infty,
\end{equation}
and that there exist constants $0<C_{1} \le C_{2} < \infty$ such that
$$
\frac{C_{1}}{m_n}\le c_{n,k}^{2} \le \frac{C_2}{m_n},
\ \  1\le k\le m_{n}.
$$
Then \eqref{eq7} holds, where $S_{n}(t)$ is defined by \eqref{eq2} with $X_{n,k}=c_{n,k}X_{k}$, $1\leq k \leq m_{n}$ if and  only if
\begin{equation}\label{eq9}
    \lim_{t\to\infty} t^{q(p,\alpha)}\,\mathbf{P}\left(\left|X_{1}\right|>t\right)=0.
\end{equation}
If $p\geq 1$ and $\alpha\in [0,1/2)\cap [0,1/p]$, then \eqref{eq7} holds as 
long as $\mathbf{E}X_1^2$ is finite.
\end{corollary}
\begin{remark}
Corollary~\ref{cor:Besov-Theorem4-Lindeberg} extends Theorem~2 of \cite{GR17} from the case of fixed random variables $(X_{k})$ to triangular arrays of random variables of the form $\bigl(c_{n,k}X_k\bigr)_{1\le k\le m_{n}}$.
\end{remark}
\section{Besov spaces framework}
First, we introduce the notation for the Faber--Schauder functions. A convenient way to parametrize this family is by using dyadic numbers directly.
We denote by $D_{j}$ the set of dyadic numbers in $[0,1]$ of level $j$, that is,
$$
D_{0}=\{0,1\}, \qquad
D_{j} =\left\{(2\ell-1)2^{-j} \; ; \; 1 \le \ell \le 2^{j-1} \right\}, \quad j \ge 1.
$$
Set
$$
D =\bigcup_{j \ge 0} D_{j}
$$
and, for $r \in D_j$, write
$$
r^{-}:=r-2^{-j}, \qquad r^{+}:=r+2^{-j}.
$$
The triangular Faber--Schauder functions $\Lambda_r$ for $r \in D_j$, $j>0$, are defined by
$$
\Lambda_r(t)=
\begin{cases}
2^j(t-r^{-}), & t \in (r^{-},r],\\[1mm]
2^j(r^{+}-t), & t \in (r,r^{+}],\\[1mm]
0, & \text{otherwise}.
\end{cases}
$$
When $j=0$, we just take the restriction to $[0,1]$:
$$
\Lambda_{0}(t)=1-t, \qquad \Lambda_{1}(t)=t, \quad t\in[0,1].
$$

\begin{theorem}[Ciesielski et al \cite{CKR93}]
Let $p>1$ and $1/p<\alpha<1$. The Faber--Schauder system $(\Lambda_{r})_{r\in D}$ is a Schauder basis
for $B^{o}_{p,\alpha}([0,1])$: each $x\in B^{o}_{p,\alpha}([0,1])$ has a unique 
representation
$$
x=\sum_{r\in D} \lambda_{r}(x)\Lambda_{r},
$$
where
$$
\lambda_{r}(x):=\frac{x(r)-x(r^{+})+x(r^{-})}{2}, \qquad r\in D_{j},\ j\ge 1,
$$
and, in the particular case $j=0$,
$$
\lambda_{0}(x):=x(0), \qquad \lambda_{1}(x):=x(1).
$$
Moreover, the norm is equivalent to the sequential norm:
$$
\|x\|_{p,\alpha} \sim \|x\|^{\mathrm{seq}}_{p,\alpha}
:= \sup_{j\ge 0} 2^{j(\alpha-1/p)}
\left(\sum_{r\in D_j} |\lambda_{r}(x)|^{p} \right)^{1/p}.
$$
\end{theorem}
The Schmidt orthogonalization procedure (with respect to the inner product in 
$\mathbb{L}_{2}(0,1)$) applied to the
Faber--Schauder system leads to the Franklin system $(f_{k})_{ k\ge 0}$:
$$
f_{k}(t)=\sum_{i=0}^{k} c_{ik}\Lambda_{r_{i}}(t), \qquad t\in [0,1],
$$
with $c_{kk}>0$ for $k\ge 0$, where the matrix $(c_{ik})$ is uniquely determined and
$(r_{i})_{i\ge 0}$ is an enumeration of $D$.

We recall a
tightness criterion in $B^{o}_{p,\alpha}([0,1])$, which is a corollary of the 
tightness criterion established in \cite{Suquet99} for 
Schauder-decomposable Banach spaces.
\begin{theorem}[Giraudo and Ra\v{c}kauskas \cite{GR17}]
\label{thm:tightness}
Let $\left(\xi_{n}\right)$ be a sequence of random processes with paths in 
$B^{o}_{p,\alpha}([0,1])$. The sequence $\left(\xi_n\right)$ is tight
if and only if the following two conditions are satisfied:
\begin{itemize}
 \item[(a)] $$
\lim_{b \to \infty} \sup_{n\geq 1}
\mathbf{P}\left(\|\xi_{n}\|_{p} > b\right) = 0;
$$
\item[(b)] for each $\varepsilon >0$,
$$
\lim_{J \to \infty} \limsup_{n\to \infty}
\mathbf{P}\left(\sup_{j \ge J}
2^{\,j(\alpha - 1/p)}\left(\sum_{r \in D_{j}}
\bigl|\lambda_{r}(\xi_{n})\bigr|^{p}\right)^{1/p} > \varepsilon \right)= 0.
$$
\end{itemize}
\end{theorem}
\section{Proofs}
\begin{proof}[Proof of Theorem~{\upshape\ref{thm:BesovThm1}}]
Fix $0\le t_{1}<\dots<t_{d}\le 1$. For each $t\in[0,1]$, let
$$
u_{n}(t):=\max\{k\in\{0,\dots,m_n\}:\ a_{n,k}\le t\},
$$
so that $a_{n,u_n(t)}\le t < a_{n,u_n(t)+1}$.
By construction of the polygonal process \eqref{eq2},
\begin{equation}\label{eq10}
S_{n}(t)=S_{n,u_n(t)}+\frac{t-a_{n,u_n(t)}}{\sigma_{n,u_n(t)+1}^2}\,X_{n,u_n(t)+1}.
\end{equation}
By the Lindeberg condition \eqref{eq4}, the Lindeberg--Feller multivariate CLT yields
$$
(S_{n}(t_{1}),\dots, S_{n}(t_d)) \xrightarrow{\mathcal D}(W(t_{1}),\dots, W(t_{d})), \ \ n\to \infty.
$$
Therefore, the finite--dimensional distributions of $\mathbb{S}_{n}$ converge to those of $\mathbb{W}$.

It remains to establish tightness of $\mathbb{S}_{n}$.
To this end, we shall use Theorem~\ref{thm:tightness}. More precisely, we need 
to verify that
\begin{equation}\label{eq11}
\lim_{b\to\infty}\sup_{n\ge1}\mathbf{P}\left(\|\mathbb{S}_{n}\|_{p}> b\right)=0;
\end{equation}
for every $\varepsilon>0$,
 \begin{equation}\label{eq12}
    \lim_{J\to\infty}\limsup_{n\to\infty}
\mathbf{P}\left(
\sup_{j\ge J}2^{j(\alpha-1/p)}\Big(\sum_{r\in D_j}|\lambda_{r}(S_n)|^{p}\Big)^{1/p}>\varepsilon
\right)=0,
 \end{equation}
where $(D_{j})_{j\ge0}$ are dyadic points, and $\lambda_{r}$ denotes the Faber--Schauder
increment at $r$.
In order to establish \eqref{eq11}, we note that due to Prokhorov invariance principle
we claim that under \eqref{eq4} and $\sum_{k=1}^{m_n}\sigma_{n,k}^{2}=1$, the polygonal
partial sum process $\mathbb{S}_{n}$ is tight in $C[0,1]$. In particular, 
$\|\mathbb{S}_n\|_{\infty}$ is tight.
Since $\|x\|_{p}\le \|x\|_{\infty}$ for every $x\in C[0,1]$ and $p\ge1$, we obtain
$$
\sup_{n \ge 1}\mathbf{P}(\|\mathbb{S}_{n}\|_{p}>b)
\le \sup_{n\ge1}\mathbf{P}(\|\mathbb{S}_{n}\|_{\infty} > b) \to 0, \ \ b \to \infty
$$
which proves \eqref{eq11}.

We now prove that condition \eqref{eq12} is satisfied. For this aim, for $r\in D_{j}$, we denote $r^{-}:=r-2^{-j}$ and $r^{+}:=r+2^{-j}$, and set
$$
\Delta_{r}:=S_{n}(r^{+})-S_{n}(r),\qquad \Delta_{r}^{-}:=S_n(r)-S_n(r^{-}).
$$
Since $2\lambda_r(S_{n})=\Delta_{r}^{-}-\Delta_{r}$, we have
$$
|\lambda_r(S_{n})|\le \frac{1}{2}|\Delta_{r}|+\frac{1}{2}|\Delta_{r}^{-}|.
$$
By symmetry, it is enough to prove \eqref{eq12} with
$|\lambda_r(S_{n})|$ replaced by $|\Delta_{r}|$. So, the proof of \eqref{eq12} reduces to proving the relation that for every
$\varepsilon>0$,
\begin{equation*} 
\lim_{J\to\infty}\limsup_{n\to\infty}
\mathbf P\left(
\sup_{j\ge J}2^{j\beta}\left(\sum_{r\in D_{j}}|\Delta_{r}|^{p}\right)^{1/p} > \varepsilon
\right)=0.
\end{equation*}
Define
$$
Y_{n}:=\max_{1\le k\le m_{n}}\sigma_{n,k}^{-2(\alpha-1/p)}|X_{n,k}|.
$$
Fix $j\ge 0$ and $r\in D_{j}$. We have to estimate the term $\Delta_{r}$. Writing $u_n(\cdot)$ for the bracket index as in
\eqref{eq10}, we distinguish three cases separately.

\smallskip
\noindent\emph{Case 1: $u_{n}(r^{+})=u_n(r)$.}
In this case, the points $r,r^{+}$ belong to the same interval
$[a_{n,u(r)},a_{n,u(r)+1}]$, and by \eqref{eq10},
$$
\Delta_{r}=\frac{r^{+}-r}{\sigma_{n,u(r)+1}^2}\,X_{n,u(r)+1}.
$$
From $r^{+}-r=2^{-j}$ and $u(r^{+})=u(r)$, it follows that
$\sigma_{n,u(r)+1}^{2} \ge 2^{-j}$.
Hence, by denoting $\gamma:=\frac{2^{-j}}{\sigma_{n,u(r)+1}^2}$ one has that $\gamma\in[0,1]$ and, since $\alpha-1/p \in (0, 1/2)$, one has
$\gamma\le \gamma^{\alpha-1/p}$. Therefore,
\begin{equation}\label{eq14}
\left|\Delta_{r}\right|
\le 2^{-(\alpha-1/p)}\sigma_{n,u(r)+1}^{-2(\alpha-1/p)}|X_{n,u(r)+1}|
\le 2^{-j(\alpha-1/p)}Y_{n}.
\end{equation}
\smallskip
\noindent\textit{Case 2: $u(r^{+})=u(r)+1$.}
Then
$$
a_{n,u(r)}\le r < a_{n,u(r)+1)}\le r^{+}<a_{n,u(r)+2}.
$$
Plugging the points $r$ and $r^{+}$ into \eqref{eq10} and taking into account the identity
$S_{n}(u(r)+1)=S_{n}(u(r))+X_{n,u(r)+1}$, we get
\begin{equation}\label{eq15}
\Delta_{r}
=\left(1-\frac{r-a_{n,u(r)}}{\sigma_{n,u(r)+1}^2}\right)X_{n,u(r)+1}
+\frac{r^{+}-a_{n,u(r)+1}}{\sigma_{n,u(r)+2}^2}\, X_{n,u(r)+2}.
\end{equation}
Note that both coefficients of above equality belong to $[0,1]$. Moreover,
\begin{multline*}
 1-\frac{r-a_{n,u(r)}}{\sigma_{n,u(r)+1}^2}
=\frac{a_{n,u(r)+1}-r}{\sigma_{n,u(r)+1}^2}
\le \frac{2^{-j}}{\sigma_{n,u(r)+1}^2}\\
\le \Bigl(\frac{2^{-j}}{\sigma_{n,u(r)+1}^2}\Bigr)^{\!\alpha-1/p}
=2^{-j(\alpha-1/p)}\sigma_{n,u(r)+1}^{-2(\alpha-1/p)},
\end{multline*}
and similarly,
$$
\frac{r^{+}-a_{n,u(r)+1}}{\sigma_{n,u(r)+2}^2}
\le \frac{2^{-j}}{\sigma_{n,u(r)+2}^2}
\le 2^{-j(\alpha-1/p)}\sigma_{n,u(r)+2}^{-2(\alpha-1/p)}.
$$
Plugging these bounds into \eqref{eq15}, we arrive at
\begin{align}
\left|\Delta_{r}\right|
& \le 2^{-j(\alpha-1/p)}\Bigl(
\sigma_{n,u(r)+1}^{-2(\alpha-1/p)}|X_{n,u(r)+1}|
+\sigma_{n,u(r)+2}^{-2(\alpha-1/p)}|X_{n,u(r)+2}|
\Bigr) \nonumber \\
& \le 2^{1-j(\alpha-1/p)}Y_{n}.\label{eq16}
\end{align}
\smallskip
\noindent\textit{Case 3: $u(r^+)\ge u(r)+2$.}
Then between points $r$ and $r^{+}$ there is at least one interval.
A telescoping argument based on \eqref{eq10} yields
\begin{equation}\label{eq17}
\Delta_{r}=
\frac{a_{n,u(r)+1}-r}{\sigma_{n,u(r)+1}^{2}}X_{n,u(r)+1}
+\sum_{k=u(r)+2}^{u(r^+)}X_{n,k}
+\frac{r^{+}-a_{n,u(r^{+})}}{\sigma_{n,u(r^{+})+1}^2}X_{n,u(r^{+})+1}.
\end{equation}
The two boundary terms in \eqref{eq17} are controlled as in Case 2, hence taking into notice relations \eqref{eq14} and \eqref{eq16}, we overall, conclude that
\begin{equation}\label{eq18}
\left|\Delta_{r}\right|
\le \Bigl|\sum_{k=u(r)+2}^{u(r^{+})}X_{n,k}\Bigr|+2^{1-j(\alpha-1/p)}Y_{n}.
\end{equation}
Furthermore, it follows from $r^{+}-r=2^{-j}$ that,
\begin{equation}\label{eq19}
\sum_{k=u(r)+2}^{u(r^{+})}\sigma_{n,k}^{2}
\le a_{n,u(r^+)}-a_{n,u(r)+1}\le r^{+}-r=2^{-j}.
\end{equation}
From \eqref{eq18} and Minkowski's inequality,
$$
\left(\sum_{r\in D_{j}}\left|\Delta_{r}\right|^{p}\right)^{1/p}
\le
\left(\sum_{r\in D_j}\Big|\sum_{k=u_n(r)+2}^{u_n(r^{+})}X_{n,k}\Big|^{p}\right)^{1/p}
+2^{1-j(\alpha-1/p)}|D_{j}|^{1/p}Y_{n}.
$$
Taking into notice that $\left|D_{j}\right| \le 2^{j}$, we get
$$
\sup_{j\ge J}2^{j\beta}\Big(\sum_{r\in D_{j}}|\Delta_{r}|^{p}\Big)^{1/p}
\le \sup_{j\ge J}2^{j(\alpha-1/p)}
\Big(\sum_{r\in D_{j}}\Big|\sum_{k=u_{n}(r)+2}^{u_n(r^{+})}X_{n,k}\Big|^{p}\Big)^{1/p}$$
$$
+ \sup_{j\ge J}2\cdot 2^{j/p}Y_{n}
=: A_{J,n}+B_{J,n}$$
Hence, for each $\varepsilon > 0$,
\begin{equation}\label{eq20}
\mathbf{P}\left(
\sup_{j\ge J}2^{j(\alpha-1/p)}\Big(\sum_{r\in D_j}|\Delta_{r}|^{p}\Big)^{1/p}> \varepsilon
\right)
\le P_{1}(J,n,\varepsilon)+P_{2}(J,n,\varepsilon),
\end{equation}
where
$$
P_{1}(J,n,\varepsilon):=\mathbf{P}(A_{J,n} > \varepsilon/2),
\qquad
P_{2}(J,n,\varepsilon):=\mathbf{P}(B_{J,n} > \varepsilon/2).
$$
It remains to prove that for each fixed $\varepsilon >0$,
\begin{equation}\label{eq21}
    \lim_{J\to \infty}\limsup_{n\to\infty}P_{1}(J,n,\varepsilon)=0,
\end{equation}
\begin{equation}\label{eq22}
   \lim_{J\to \infty}\limsup_{n\to\infty}P_{2}(J,n,\varepsilon)=0.
\end{equation}
First, we prove \eqref{eq22}. It is clear that
$$
P_2(J,n,\varepsilon) \leq \mathbf{P} \left(Y_{n}>\frac{\varepsilon}{4}\,2^{-J/p}\right)
\le \sum_{k=1}^{m_{n}}\mathbf P\big(|X_{n,k}|>\frac{\varepsilon}{4}2^{-J/p} \sigma_{n,k}^{2(\alpha-1/p)}\big).
$$
In view of condition \eqref{eq5}, the right-hand side of the above inequality vanishes as $n\to \infty$. This proves \eqref{eq22}.

We now prove that \eqref{eq21} holds. For this, we define the truncated random variables, that is, for each $\tau\in(0,1)$,
$$
X_{n,k,\tau}:=X_{n,k}\mathbf{1}{\{|X_{n,k}|\le \tau\sigma_{n,k}^{2(\alpha-1/p)}\}},
\qquad
\widetilde X_{n,k,\tau}:=X_{n,k,\tau}-\mathbf{E} X_{n,k,\tau}.
$$
Let us consider event $A_{n,\tau}:=\{Y_n\le \tau\}=\{\max_{1\leq k\leq m_n} |X_{n,k}|\le \tau\sigma_{n,k}^{2(\alpha-1/p)}\}$.
By condition of Theorem 1 with $\varepsilon=\tau$,
$$
\mathbf{P}(A_{n,\tau}^{c})=\mathbf{P}(Y_{n} >\tau)
\le \sum_{k=1}^{m_n}\mathbf{P}\left(\left|X_{n,k}\right|>\tau\sigma_{n,k}^{2(\alpha-1/p)}\right)\to 0, \ \ n\to \infty.
$$
It is easy to notice that on the events $A_{n,\tau}$, one has $X_{n,k}=X_{n,k,\tau}$ for all $k \geq 1$ and hence,
$$
P_{1}(J,n,\varepsilon)\le P_{1,\tau}(J,n,\varepsilon)+\mathbf{P}(A_{n,\tau}^{c}),
$$
where $P_{1,\tau}(J,n,\varepsilon)$ is defined as $P_1(J,n,\varepsilon)$ but 
with $X_{n,k}$ replaced by $X_{n,k,\tau}$.\\
Therefore, the proof of \eqref{eq22} reduces to prove that 
\begin{equation}\label{eq23}
\lim_{\tau \to 0}\lim_{J\to\infty}\limsup_{n\to\infty} P_{1,\tau}(J,n,\varepsilon)=0.
\end{equation}
To this aim, we represent $X_{n,k,\tau}$ as $X_{n,k,\tau}=\mathbf{E}X_{n,k,\tau}+\widetilde X_{n,k,\tau}$.
Since $X_{n,k}$ are centered, we have
$$
\mathbf{E} X_{n,k,\tau}= -\mathbf{E}\big(X_{n,k}\mathbf{1}{\{|X_{n,k}|>\tau\sigma_{n,k}^{2(\alpha-1/p)}\}}\big),
$$
and by the Cauchy--Schwarz inequality, one has
\begin{equation*}
 \left|\mathbf{E} X_{n,k,\tau}\right|
\le 
\sigma_{n,k}\,\left(\mathbf{P}\left(|X_{n,k}|>\tau\sigma_{n,k}^{2(\alpha-1/p)}
\right)\right)^{1/2}.
\end{equation*}
Applying the Cauchy--Schwarz inequality once again, we infer from \eqref{eq19} that for each block $I_{r}:=\{u_n(r)+2,\dots,u_{n}(r^{+})\}$,
$$
\Big|\sum_{k\in I_{r}}\mathbf{E} X_{n,k,\tau}\Big|
\le 2^{-j/2}\Big(\sum_{k=1}^{m_n}\mathbf P(|X_{n,k}|>\tau\sigma_{n,k}^{2(\alpha-1/p)})\Big)^{1/2}.
$$
Multiplying the right-hand side of the above inequality by $2^{j(\alpha-1/p)}$
yields the factor $2^{-j(1/2-(\alpha-1/p))}$, which is summable over $j \ge J$
since $1/2-(\alpha-1/p)>0$; moreover, the second factor tends to $0$ by virtue
of condition~\eqref{eq5}.
Thus, the contribution of the means is negligible in \eqref{eq23},
and we may deal with centered blocks based on $\widetilde X_{n,k,\tau}$ as 
follows. Let $q$ be such that $q(p,\alpha)<q<p$. Then
\begin{multline*}
P_{1,\tau}(J,n,\varepsilon)\\ 
\leq 
\sum_{j\geq J}\mathbf P
\left(2^{j(\alpha-1/p)}\left(\sum_{r\in D_j}
\left\lvert \sum_{k\in I_r}\widetilde X_{n,k,\tau} \right\rvert^p 
\right)^{1/p}>\varepsilon/4\right)\quad \mbox{(by a union bound)}  \\
 \leq \frac{4^q}{\varepsilon^q}
\sum_{j\geq J}2^{jq(\alpha-1/p)}
\mathbf{E}\left[ \left(\sum_{r\in D_j}
\left\lvert \sum_{k\in I_r}\widetilde X_{n,k,\tau} \right\rvert^p \right)^{q/p} 
\right]\quad \mbox{(by Markov's inequality)}\\
 \leq  \frac{4^q}{\varepsilon^q}
\sum_{j\geq J}2^{jq(\alpha-1/p)}
\sum_{r\in D_j}\mathbf{E} 
\left\lvert \sum_{k\in I_r}\widetilde X_{n,k,\tau} \right\rvert^q   \quad 
\mbox{(by concavity of the map } t\mapsto t^{q/p})  \\
 \leq C_q
\frac{4^q}{\varepsilon^q}
\sum_{j\geq J}2^{jq(\alpha-1/p)}
\sum_{r\in D_j} \left( \sum_{k\in I_r}\mathbf{E} \widetilde 
X_{n,k,\tau}^2 \right)^{q/2}\nonumber\\
+ C_q
\frac{4^q}{\varepsilon^q}
\sum_{j\geq J}2^{jq(\alpha-1/p)}\sum_{k\in I_r}
\mathbf{E} 
\left\lvert  \widetilde X_{n,k,\tau} \right\rvert^q   \quad \mbox{(by 
Rosenthal's inequality \cite{Rosenthal70})},
\end{multline*}
where $C_q$ depends only on $q$. After a use of the elementary 
inequalities 
$\mathbf{E}\widetilde X_{n,k,\tau}^{2}\le \sigma_{n,k}^{2}$, 
$\mathbf{E}|\widetilde X_{n,k,\tau}|^{q}\le 
2^{q-1}\mathbf{E}|X_{n,k,\tau}|^{q}$, we derive that 
\begin{multline}\label{eq:estimation_P_with_X_tilde}
 P_{1,\tau}(J,n,\varepsilon)\leq C_q
\frac{4^q}{\varepsilon^q}
\sum_{j\geq J}2^{jq(\alpha-1/p)}
\sum_{r\in D_j} \left( \sum_{k\in I_r}\mathbf{E} 
X_{n,k}^2 \right)^{q/2} \\
 + C_q
\frac{2^{3q-1}}{\varepsilon^q}
\sum_{j\geq J}2^{jq(\alpha-1/p)}\sum_{r\in D_j}\sum_{k\in I_r}
\mathbf{E} 
\left\lvert   X_{n,k,\tau} \right\rvert^q  
 .
\end{multline}
Using $$ \sum_{k\in I_r}\mathbf{E}X_{n,k}^2 =\sum_{k=u_n(r)+2}^{u_n(r^+)}
\sigma_{n,k}^2\leq 2^{-j}$$ (see the first line page 430 of \cite{RS03}) and 
$\operatorname{Card}(D_j)\leq 2^j$ for the first term in 
\eqref{eq:estimation_P_with_X_tilde} and switching the sums over $j$ and $k$ for the second gives 
\begin{multline}\label{eq:estimation_P_with_X_tilde2}
P_{1,\tau}(J,n,\varepsilon)\leq 
 C_q
\frac{4^q}{\varepsilon^q}
\sum_{j\geq J}2^{j\left(q(\alpha-1/p-1/2)+1\right)}
\\
+C_q\frac{2^{3q-1}}{\varepsilon^q}
\sum_{k=1}^{m_n}
\sum_{j=1}^{\infty}2^{jq(\alpha-1/p)}\sum_{r\in D_j}\mathbf{1}
\{u_n(r)+2\leq 
k\leq u_n(r^+)\}
\mathbf{E} 
\left\lvert   X_{n,k,\tau} \right\rvert^q   .
\end{multline}
Using that 
$$
\mathbf{1}\{u_n(r)+2\leq k\leq u_n(r^+)\}
\leq \mathbf{1}\{\sigma_{n,k}^2\leq 2^{-j}\}
$$
(see the displayed equation before (31) in \cite{RS03}), the bound \eqref{eq:estimation_P_with_X_tilde2} becomes 
\begin{multline}\label{eq:estimation_P_with_X_tilde3}
P_{1,\tau}(J,n,\varepsilon)\leq 
 C_q
\frac{4^q}{\varepsilon^q}
\sum_{j\geq J}2^{j\left(q(\alpha-1/p-1/2)+1\right)}
\\
+C_q\frac{2^{3q-1}}{\varepsilon^q}
\sum_{k=1}^{m_n}
\sum_{j=1}^{\infty}2^{jq(\alpha-1/p)}\mathbf{1}\{\sigma_{n,k}^2\leq 2^{-j}\}
\mathbf{E} 
\left\lvert   X_{n,k,\tau} \right\rvert^q .
\end{multline}
The inner sum over $j$ can be estimated as  
\begin{align*}
\sum_{j=1}^{\infty}2^{jq(\alpha-1/p)}\mathbf{1}\{\sigma_{n,k}^2\leq 2^{-j}\}
&=\sum_{j=1}^{\log_2\left(\sigma_{n,k}^{-2}\right)}2^{jq(\alpha-1/p)}\\
&\leq 
\frac{2^{q(\alpha-1/p)\log_2\left(\sigma_{n,k}^{-2}\right)}}{2^{q(\alpha-1/p)}-1
}\\
&=\frac{ \sigma_{n,k}^{-2q(\alpha-1/p)}}{2^{q(\alpha-1/p)}-1}.
\end{align*}
As a consequence, the bound \eqref{eq:estimation_P_with_X_tilde3} becomes
\begin{multline*}
P_{1,\tau}(J,n,\varepsilon)\leq 
 C_q
\frac{4^q}{\varepsilon^q}
\sum_{j\geq J}2^{j\left(1-q/q(p,\alpha)\right)}\\
+\frac{2^{3q-1}C_q}{\varepsilon^q(2^{q(\alpha-1/p)}-1)}
\sum_{k=1}^{m_n}
 \sigma_{n,k}^{-2q(\alpha-1/p)}\mathbf{E}\left[\lvert 
X_{n,k}\rvert^q\mathbf{1}\{\lvert X_{n,k}\rvert\leq 
\tau\sigma_{n,k}^{2(\alpha-1/p} \} \right].
 \end{multline*}
Using that $q>q(p,\alpha)$ and \eqref{eq6}
we conclude that \eqref{eq23} holds hence \eqref{eq21}.

Combining relations \eqref{eq20}, \eqref{eq21} with \eqref{eq22} we infer in 
view of \eqref{eq12} that $\mathbb{S}_{n}$ is tight in 
$B^{o}_{p,\alpha}([0,1])$.
This ends the proof of Theorem~\ref{thm:BesovThm1}.
\end{proof}
\begin{proof}[Proof of Theorem~\ref{thm:WIP_p_alpha_leq_1}]
 As in \cite{GR17}, using the continuous embeddings 
$B_{2,\alpha}^o([0,1])\hookrightarrow B_{p,\alpha}^o([0,1])$ if $1\leq p\leq 2$ 
and 
$0\leq\alpha<1/2$, it suffices to prove the following cases 
\begin{itemize}
 \item $p=2$ and 
$0\leq \alpha<1/2$;
\item $p>2$ and $0\leq\alpha\leq 1/p$.
\end{itemize}
The convergence of the finite dimensional distributions is guaranted by 
Lindeberg condition \eqref{eq4}. It thus suffices to show tightness, which 
translates as 
\begin{equation}\label{eq:tightness_p_alpha_leq_1}
 \forall \varepsilon>0,\quad \lim_{\delta\to 0}
 \limsup_{n\to\infty}\mathbf P\left(\frac 1{\delta^{p\alpha}}
 \sup_{h:\abs{h}\leq \delta} \int_{I_h}
\abs{S_n(t+h)-S_n(t)}^p\mathrm{d}t >\varepsilon 
  \right)=0,
\end{equation}
where $I_h=[0,1]\cap [-h, 1,1-h]$. The idea is the following. First consider the 
case where $p=2$ (and $0<\alpha<1/2$). We bound 
the supremum involved in \eqref{eq:tightness_p_alpha_leq_1} by a function of 
the squares of increments of the partial sums of the $X_{n,i}$. Bounding them further by maxima of increments and using Doob's inequality gives \eqref{eq:tightness_p_alpha_leq_1}. 

In the second case, that is, $p>2$, the previous argument does not work out directly. Indeed, we cannot use Markov's inequality, since we do not assume that the random variables $X_{n,i}$ have a finite moment of order $p$. To overcome this problem, we 
will split the probability in \eqref{eq:tightness_p_alpha_leq_1} according to 
the cases where $ \sup_{h:\abs{h}\leq \delta} \sup_{t\in 
I_h}\abs{S_n(t+h)-S_n(t)}^{p-2}$ is bigger or not than a fixed value. The part 
where this random variable is bigger than the fixed number is handled by 
tightness in the space of continuous functions. On the part where it is smaller, 
 the $\mathbb L_p$-norm involved in \eqref{eq:tightness_p_alpha_leq_1} reduces 
to an $\mathbb L_2$-norm, for which the argument in the case $p=2$ applies.

Let us start by the case $p=2$. Define the intervals  
$J_{n,k}:=[a_{n,k-1},a_{n,k})$, $1\leq k\leq m_n$. 
For a fixed $\delta\in(0,1)$, $h$ such that $\abs{h}\leq \delta$, one has 
\begin{equation*}
\int_{I_h} \left(S_n(t+h)-S_n(t) \right)^2\mathrm{d}t\leq 
\sum_{k,\ell=1}^{m_n}\int_{I_h}\mathbf{1}\{t\in J_{n,k}\}
\mathbf{1}\{t+h\in J_{n,\ell}\}
 \left(S_n(t+h)-S_n(t) \right)^2\mathrm{d}t.
\end{equation*}
Using twice relation \eqref{eq2}, first with $k$ and $t$, then with $\ell$ and $t+h$
and the inequalities 
$\abs{S_n(t+h)-S_n(t)}\leq \abs{S_n(t+h)-S_n(a_{n,\ell})}
+ \abs{ S_n(a_{n,\ell})-S_n(a_{n,k})}+\abs{S_n(a_{n,k}) -S_n(t)}$, 
$(a+b+c)^2\leq 3a^2+3b^2+3c^2$ valid for all real numbers $a,b$ and $c$, one 
gets 
\begin{align*}
\int_{I_h} \left(S_n(t+h)-S_n(t) \right)^2\mathrm{d}t&\leq 
3\sum_{k,\ell=1}^{m_n}\int_{I_h}\mathbf{1}\{t\in J_{n,k}\}
\mathbf{1}\{t+h\in 
J_{n,\ell}\}\left(\frac{t+h-a_{n,\ell}}{\sigma_{n,\ell}^2}\right)^2  
 X_{n,\ell}^2\mathrm{d}t 
 \\
&+3\sum_{k,\ell=1}^{m_n}\int_{I_h}\mathbf{1}\{t\in J_{n,k}\}
\mathbf{1}\{t+h\in 
J_{n,\ell}\}\left(\sum_{i=\min\{k,\ell\}+1}^{\max\{k,\ell\}}X_{n,i}
 \right)^2\mathrm{d}t   \\
&+3\sum_{k,\ell=1}^{m_n}\int_{I_h}\mathbf{1}\{t\in J_{n,k}\}
\mathbf{1}\{t+h\in J_{n,\ell}\}\left(\frac{t-a_{n,k}}{\sigma_{n,k}^2}\right)^2 
X_{n,k}^2\mathrm{d}t.
\end{align*}
For the first term (respectively the third), the sum over $k$ (respectively $\ell$) does not play any role since we sum indicators of disjoint sets whose union is contained in $[0,1]$.
Let 
\begin{equation*}
E_{n,k,\delta}:=\{\ell\in\{1,\dots,m_n\}, \abs{a_{n,\ell}-a_{n,k}}\leq \delta\}.
\end{equation*}
 Noticing that  $\mathbf{1}\{t\in J_{n,k}\}
\mathbf{1}\{t+h\in J_{n,\ell}\}=0$ if $\ell\notin E_{n,k,\delta}$,
writing  $E_{n,k,\delta}$ as 
$\{i\in\{1,\dots,m_n\}, j_{n,k,\delta}\leq i\leq j'_{n,k,\delta}\}$,
and using
\begin{equation*}
 \left(\sum_{i=\min\{k,\ell\}+1}^{\max\{k,\ell\}}X_{n,i}
 \right)^2\leq \max_{j\in 
E_{n,k,\delta}}\left(\sum_{i=\min\{k,j\}+1}^{\max\{k,j\}}X_{n,i}
 \right)^2\leq 4\max_{1\leq j\leq j'_{n,k,\delta}}\left(
 \sum_{i=j_{n,k,\delta}+1}^jX_{n,i}
 \right)^2,
\end{equation*}
one gets 
\begin{multline*}
\int_{I_h} \left(S_n(t+h)-S_n(t) \right)^2\mathrm{d}t\leq 
2 \sum_{k=1}^{m_n} \sigma_{n,k}^2X_{n,k}^2\\
+8\sum_{k,\ell=1}^{m_n} \lambda(J_{n,k}\cap (J_{n,\ell}-h))
\max_{1\leq j\leq j'_{n,k,\delta}}\left(
 \sum_{i=j_{n,k,\delta}+1}^jX_{n,i}
 \right)^2,
\end{multline*}
where $\lambda$ denotes Lebesgue measure and $J_{n,\ell}-h=\{u-h,u\in 
J_{n,\ell}\}$. 
Since the sets $J_{n,k}\cap (J_{n,\ell}-h)$, $1\leq \ell\leq m_n$, are pairwise 
disjoint and their union has a Lebesgue measure smaller than $\delta$, we 
finally get the bound 
\begin{multline}\label{eq:bound_supremum_p_alpha_leq_1}
\sup_{h:\abs{h}\leq \delta}\int_{I_h} \left(S_n(t+h)-S_n(t) \right)^2\mathrm{d}t
\\
\leq 2 \sum_{k=1}^{m_n} \sigma_{n,k}^2X_{n,k}^2
+8\sum_{k=1}^{m_n}\sigma_{n,k}^2
\max_{1\leq j\leq j'_{n,k,\delta}}\left(
 \sum_{i=j_{n,k,\delta}+1}^jX_{n,i}
 \right)^2.
\end{multline}
In order to conclude the case $p=2$, it suffices to prove that 
\begin{center}
\begin{equation}\label{eq:conv_proba_powers} 
 \sum_{k=1}^{m_n} \sigma_{n,k}^2X_{n,k}^2\to 0\mbox{ in 
probability}
\end{equation}
\begin{equation}
 \lim_{\delta\to 0}\limsup_{n\to\infty}
\mathbf{P}\left(\frac 1{\delta^{2\alpha}}\sum_{k=1}^{m_n}\sigma_{n,k}^2
\max_{1\leq j\leq j'_{n,k,\delta}}\left(
 \sum_{i=j_{n,k,\delta}+1}^jX_{n,i}
 \right)^2>\varepsilon \right)=0\label{eq:conv_increments}.
\end{equation}
\end{center}
For \eqref{eq:conv_proba_powers}, we use Markov's inequality and the fact that 
$\max_{1\leq k\leq m_n}\sigma_{n,k}^2\to 0$ as $n\to \infty$. For 
\eqref{eq:conv_increments}, splitting the max according to $j\leq k$ or not, 
using Markov's and Doob's inequality,  one gets
\begin{align*}
\mathbf{E}
\left(\frac 1{\delta^{2\alpha}}\sum_{k=1}^{m_n}\sigma_{n,k}^2
\max_{1\leq j\leq j'_{n,k,\delta}}\left(
 \sum_{i=j_{n,k,\delta}+1}^jX_{n,i}
 \right)^2  \right)&\leq 2{\delta^{-2\alpha}}\sum_{k=1}^{m_n}\sigma_{n,k}^2
  \sum_{i= j_{n,k,\delta}+1}^{ j'_{n,k,\delta}}\sigma_{n,i}^2\\
  &\leq 2\delta^{1-2\alpha}\sum_{k=1}^{m_n}\sigma_{n,k}^2= 2\delta^{1-2\alpha}
\end{align*}
and we conclude using $\alpha<1/2$. 

Let us now treat the case $p>2$. Define the event
\begin{equation*}
A\left(\delta,n,\tau\right):=\left\{ \sup_{h:\abs{h}\leq \delta} \sup_{t\in I_h}\abs{S_n(t+h)-S_n(t)}^{p-2} >\tau \right\}.
\end{equation*}
Since the process $\mathbb{S}_n$ is tight in the space of 
continuous functions on $[0,1]$ endowed with the uniform norm, we have
\begin{equation*}
 \forall \tau>0, \quad \lim_{\delta\to 0}\limsup_{n\to \infty}\pr 
\left(A\left(\delta,n,\tau\right)\right)=0 .
\end{equation*}
Therefore, it suffices to show that 
\begin{equation*}
 \forall \varepsilon>0,\quad   \lim_{\tau\to 0}\lim_{\delta\to 0}
 \limsup_{n\to\infty}\mathbf P\left(A\left(\delta,n,\tau\right)^c\cap B(n,\delta,\varepsilon)
  \right)=0,
\end{equation*}
where
\begin{equation*}
B(n,\delta,\varepsilon)=\left\{ \frac 1{\delta^{p\alpha}}
 \sup_{h:\abs{h}\leq \delta} \int_{I_h}
\abs{S_n(t+h)-S_n(t)}^p\mathrm{d}t >\varepsilon \right\}.
\end{equation*}
Using $\int_{I_h} \lvert g(t)\rvert^p\mathrm{d}t\leqslant \sup_{s\in I_h}
\lvert g(s)\rvert^{p-2}\int_{I_h} \lvert g(t)\rvert^2\mathrm{d}t$ 
followed by \eqref{eq:bound_supremum_p_alpha_leq_1}, one has 
\begin{align*}
A\left(\delta,n,\tau\right)^c\cap B(n,\delta,\varepsilon)
&\subset \left\{ \frac 1{\delta^{p\alpha}}
 \sup_{h:\abs{h}\leq \delta} \int_{I_h}
\abs{S_n(t+h)-S_n(t)}^2\mathrm{d}t >\frac{\varepsilon}{2\tau } \right\}\\
&\subset \left\{\frac 2 {\delta^{p\alpha}} \sum_{k=1}^{m_n} 
\sigma_{n,k}^2X_{n,k}^2>\frac{\varepsilon}{2\tau }\right\}\\
&\cup\left\{8\sum_{k=1}^{m_n}\sigma_{n,k}^2
\max_{1\leq j\leq j'_{n,k,\delta}}\left(
 \sum_{i=j_{n,k,\delta}+1}^jX_{n,i}
 \right)^2>\frac{\varepsilon}{ \tau } \right\}  .
\end{align*}
We have already seen that for each fixed positive $\delta, \tau$, 
\begin{equation*}
 \lim_{n\to\infty}\mathbf{P}\left(
\frac 2 {\delta^{p\alpha}} \sum_{k=1}^{m_n} 
\sigma_{n,k}^2X_{n,k}^2>\frac{\varepsilon}{2\tau }\right)=0.
\end{equation*}
Using Markov's 
inequality and similar arguments as before, one gets that 
\begin{align*}
\limsup_{n\to\infty}\mathbf{P}\left(16\sum_{k=1}^{m_n}\sigma_{n,k}^2
\max_{1\leq j\leq j'_{n,k,\delta}}\left(
 \sum_{i=j_{n,k,\delta}+1}^jX_{n,i}
 \right)^2>\frac{\varepsilon}{ 2\tau } \right)\leq  32\: 
\delta^{1-p\alpha}\tau 
.
\end{align*}
If $p\alpha<1$, this quantity goes to $0$ as $\delta\to 0$ and if $p\alpha=1$, 
the quantity vanishes as $\tau\to 0$. This ends the proof of 
Theorem~\ref{thm:WIP_p_alpha_leq_1}.
\end{proof}
\begin{proof}[Proof of Corollary~\ref{cor:Besov-Theorem4-Lindeberg}]
That the conditions are sufficient follows the same lines 
as in the proof of Theorem~4 in  \cite{RS03}. For necessity, 
we  proceed as follows: if the aforementioned weak convergence holds, 
then for each $\eta$, $\varepsilon>0$, we can find $\delta$ and $n_0$ such that 
for
$n\geq n_0$, 
\[
 \mathbf P\left(\frac 1{\delta^{p\alpha}}
 \sup_{h:\abs{h}\leq \delta} \int_{I_h}
\abs{S_n(t+h)-S_n(t)}^p\mathrm{d}t >\varepsilon
  \right) \leq\eta.
\]
For $n$ such that $n\geq n_0$ and $C_2/m_n\leq \delta/2$, 
we bound the integral from below by splitting $I_h$ 
into the intervals $[a_{n,k-1},a_{n,k}-\delta]$ so that if 
$t$ belongs 
to $[a_{n,k-1},a_{n,k}]$ then so does $t+h$. Bounding $\delta$ by a constant 
times $m_n^{-1}$ 
shows that the sequence $\left(m_n^{-p/q(p,\alpha)} 
\sum_{i=1}^{m_n}\abs{X_k}^p
\right)_{n\geq 1}$ converges to $0$ in probability as $n$ goes to infinity and 
so does the sequence  
$\left(m_n^{-1/q(p,\alpha)} 
\max_{1\leq i\leq m_n}\abs{X_k} 
\right)_{n\geq 1}$. Using elementary inequalities on tails of independent 
random variables, this implies that 
$m_n^{q(p,\alpha)}\mathbf{P}(\abs{X_1}> m_n)\to 0$ as $n$ goes to infinity. 
The convergence 
\eqref{eq9} follows by bounding the term for $t\in [m_n,m_{n+1})$ 
using \eqref{eq:condition_on_mn}.
\end{proof}
 
\textbf{Acknowledgements.} 
This work was carried out during the second author's visit to the University of 
Strasbourg. The authors gratefully acknowledge that this work was partially 
supported by a grant from the International Mathematical Union's Commission for 
Developing Countries (IMU-CDC) and the Simons Foundation.

   \end{document}